\documentclass[a4paper,twoside,12pt]{article}
\usepackage{amssymb,amsfonts,amsmath,amsthm,latexsym}
\usepackage[pdftex,bookmarks,colorlinks=false]{hyperref}
\usepackage[hmargin=1.2in,vmargin=1.2in]{geometry}
\usepackage[auth-sc-lg,affil-sl]{authblk}
\usepackage{graphicx,caption,subcaption}
\usepackage[mathscr]{euscript}
\newtheorem{thm}{Theorem}[section]

\newtheorem{defn}[thm]{Definition}

\newtheorem{prop}[thm]{Proposition}
\newtheorem{rem}[thm]{Remark}

\newtheorem*{ill}{Illustration}

\numberwithin{equation}{section}

\def\ni{\noindent}
\def\N{\mathbb{N}}
\def\G{\mathscr{G}}
\def\cS{\mathcal{S}}
\def\cE{\mathcal{E}}
\def\cP{\mathcal{P}}
\def\fE{\mathfrak{E}}
\def\fS{\mathfrak{S}}
\def\fM{\mathfrak{M}}

\title{\textbf{\sc  A Study on Edge-Set Graphs of Certain Graphs}}

\author{Johan Kok}
\affil{\small Tshwane Metropolitan Police Department\\ City of Tshwane, Republic of South Africa \\ {\tt kokkiek2@tshwane.gov.za}}

\author{N. K. Sudev}
\affil{\small Department of Mathematics\\ Vidya Academy of Science \& Technology \\ Thalakkottukara, Thrissur - 680501, India.\\ {\tt sudevnk@gmail.com}}

\author{K. P. Chithra}
\affil{\small Naduvath Mana, Nandikkara \\ Thrissur - 680301, India.\\ {\tt chithrasudev@gmail.com}}

\date{}

\begin{document}
\maketitle

\begin{abstract}
Let $G(V,E)$ simple connected graph, with $|E|=\epsilon$. In this paper, we define an edge-set graph  $\G_G$ constructed from the graph $G$ such that any vertex $v_{s,i}$ of $\G_G$ corresponds to the $i$-th $s$-element subset of $E(G)$ and any two vertices $v_{s,i}, v_{k,m}$ of $\G_G$ are adjacent if and only if there is at least one edge in the edge-subset corresponding to $v_{s,i}$ which is adjacent to at least one edge in the edge-subset corresponding to $v_{k,m}$, where $s,k$ are positive integers. It can be noted that the  edge-set graph $\G_G$ of a graph $G$ is dependent on both the structure of $G$ as well as the number of edges $\epsilon$. We also discuss the characteristics and properties of the edge-set graphs corresponding to certain standard graphs.
\end{abstract}

\ni \textbf{Keywords:} Edge-set graph, total edge-degree of a graph, edge-degree of vertex, connected edge dominating set, artificial edge-set element.

\vspace{0.2cm}

\ni\textbf{Mathematics Subject Classification:} 05C07, 05C38, 05C78

\section{Introduction}

For general notation and concepts in graph theory, we refer to \cite{BM1}, \cite{CL1}, \cite{GY1}, \cite{DBW}. All graphs mentioned in this paper are simple, connected finite and undirected graphs unless mentioned otherwise. 

Let $\N$ denotes the set of all positive integers. Then, the notion of a set-graph has been introduced in \cite{KC2S} associated with a given non-empty set, $A^{(n)} = \{a_1, a_2, a_3, \ldots, a_n\}, n \in \N$ as follows.

\begin{defn}\label{D-SG}{\rm 
\cite{KC2S} Let $A^{(n)} = \{a_1, a_2, a_3, \ldots, a_n\}, n\in \N$ be a non-empty set and the $i$-th $s$-element subset of $A^{(n)}$ be denoted by $A_{s,i}^{(n)}$.  Now consider $\cS = \{ A_{s,i}^{(n)}: A_{s,i}^{(n)} \subseteq A^{(n)}, A_{s,i}^{(n)} \ne \emptyset\}$. The \textit{set-graph} corresponding to set $A^{(n)}$, denoted $G_{A^{(n)}}$, is defined to be the graph with $V(G_{A^{(n)}}) = \{v_{s,i}: A_{s,i}^{(n)} \in \cS\}$ and $E(G_{A^{(n)}}) = \{v_{s,i}v_{t,j}:~ A_{s,i}^{(n)}\cap A_{t,j}^{(n)}\ne \emptyset\}$, where $s\ne t~ \text{or}~ i\ne j$. }
\end{defn} 

A detailed study on set-graphs has also been done in \cite{KC2S}.  Some of the major results proved in that paper, which are relevant in our present study, are mentioned in this section.

\begin{thm}\label{Thm1.2}
{\rm \cite{KC2S}} Let $G_{A^{(n)}}$ be a set-graph. Then, the vertices $v_{s,i}, v_{s,j}$ of $G_{A^{(n)}}$, corresponding to subsets $A^{(n)}_{s,i}$ and $A^{(n)}_{s,j}$ in $\cS$ of equal cardinality, have the same degree in $G_{A^{(n)}}$.
\end{thm}

In other words, we can say that the vertices of a set-graph $G_{A^{(n)}}$ corresponding to the equivalent subsets (subsets having same cardinality) of a non-empty set $A^{(n)}$ have the same degree.

The degree of a vertex in a set-graph has been determined in \cite{KC2S} using the principle of inclusion and exclusion as follows.
 
\begin{thm}\label{T-DSG1}
{\rm \cite{KC2S}} Let $G = G_{A^{(n)}}$ be a set-graph in respect of a non-empty set $A^{(n)} = \{a_1, a_2, a_3, \ldots, a_n\}, n \in \N$ and let $v_{k,i}$ be an arbitrary vertex of $G$ corresponding to a $k$-element subset of $A^{(n)}$. Then, $d_G(v_{k,i}) = (\sum\limits_{J}(-1)^{|J|-1}\cdot |\bigcap\limits_{j\in J} \cS_j|) -1$, where $J$ is an indexing set such that $\emptyset \ne J \subseteq \{0, 1, 2, \ldots, k\}$ and $\cS_j$ is the collection of subsets of $A^{(n)}$ containing the element $a_j$.
\end{thm}

The following is an important result regarding the minimal and maximal degree of vertices in a set-graph and the relation between them.

\begin{thm}\label{T-DSG2}
{\rm \cite{KC2S}} For any vertex $v_{s,i}$ of a set-graph $G = G_{A^{(n)}}$, we have $2^{n-1}-1 \le d_G(v_{s,i}) \le 2(2^{n-1}-1)$. Moreover, for any set-graph $G = G_{A^{(n)}}, \Delta (G) = 2\delta (G)$.
\end{thm}

It was also mentioned in \cite{KC2S} that there exists a unique vertex $v_{n,1}$ having degree $\Delta (G)$. Furthermore, it can be seen that $\Delta (G)$ is always an even number whilst $\delta (G)$ is always an odd number.

Motivated from this study of set-graphs, in this paper, we introduce the notion of an edge-set graph derived from a given non-empty non-trivial graph $G$ and study the characteristics and structural properties of the edge-set graphs of certain standard graphs.

\section{Edge-Set Graph of a Graph}

Let $A$ be a non-empty finite set. We denote $\cS$ be the collection of all $s$-element subsets of $A$ (arranged in some order), where $1\le s\le |A|$, and the $i$-th element $\cS$ by $A_{i,s}$. In this paper, we study certain graphs whose vertices are labeled by distinct elements of the collection $\mathcal{P}(A)-\{\emptyset\}$ in an injective manner. We denote the vertex of a graph corresponding to the set $A_{s,i}$  by $v_{s,i}$

Using these concepts, similar to the notion of a set-graph, let us first define the notion of an edge-set graph of a given graph $G$ as follows.

\begin{defn}\label{D-ESG}{\rm 
Let $G(V,E)$ be a non-empty finite graph with $|E|=\epsilon$ and $\cE=\cP(E)-\{\emptyset\}$, where $\cP(E)$ is the power set of the edge set $E(G)$. For $1\le s\le \epsilon$, let $\cS$ be the collection of all $s$-element subsets of $E(G)$ and $E_{s,i}$ be the $i$-th element of $\cS$. Then, the \textit{edge-set graph} corresponding to $G$, denoted by $G_{E^{(\epsilon)}}$ or $\G_G$, is the graph with the following properties.
\begin{enumerate}\itemsep0mm
\item[(i)] $|V(\G_G)|=2^{\epsilon}-1$ so that there exists a one to one correspondence between $V(\G_G)$ and $\cE$.
\item[(ii)] Two vertices, say $v_{s,i}$ and $v_{t,j}$, in $\G_G$ are adjacent if some elements (edges of $G$) in $E_{s,i}$ is adjacent to some elements of $E_{t,j}$ in $G$.
\end{enumerate}}
\end{defn} 

From the above definition, it can be seen that the edge-set graph $\G_G$ of a given graph $G$ is dependent not only on the number of edges $\epsilon$, but the structure of $G$ also. Therefore, we have distinct edge-set graphs for non-isomorphic graphs of the same size. 

As in the case of a set-graph $G_{A^{(n)}}$, it can be observed that an edge-set graph $\G_G$ has an odd number of vertices. It can also be noted that if $G$ is a trivial graph, then $\G_G$ is an empty graph (since $\epsilon=0$). Also, we have $\G_{P_2} = K_1$ and $\G_{P_3}=C_3$.

In this paper, we use the following conventions.
\begin{enumerate}\itemsep0mm
\item[(i)] If an edge $e_j$ is incident with vertex $v_k$, then we write it as $(e_j \to v_k)$.
\item [(ii)] If the edges $e_i$ and $e_j$ of a graph $G$ are adjacent, then we write it as $e_i\cong e_j$.
\item[(iii)] The $n$ vertices of the path $P_n$ are positioned horizontally and the vertices and edges are labeled from left to right as $v_1, v_2, v_3, \ldots, v_n$ and $e_1, e_2, e_3, \ldots, e_{n-1}$, respectively. 
\item[(iv)] The $n$ vertices of the cycle $C_n$ are seated on the circumference of a circle and the vertices and edges are labeled clockwise as $v_1, v_2, v_3, \ldots, v_n$ and $e_1, e_2, e_3, \ldots, e_n$, respectively such that $e_i=v_iv_{i+1}$, in the sense that $v_{n+1}=v_1$.
\end{enumerate} 

Analogous to the degree of a vertex in a given graph $G$, we define the edge-degree of elements (vertices or edges) in $G$ as follows. 

\begin{defn}\label{Def-2.2}{\rm
For a graph $G$ with edges $E(G) = \{e_i: 1 \le i \le \epsilon\}$, the \textit{edge-degree of an edge} $e_i$ of $G$ incident with vertex $v_k$ is equal to the number of edges $e_j \to v_k, e_i \ne e_j$ and  denoted by $d_{G(v_k)}(e_i)$. If the edge $e_i$ is incident with vertices $v_k, v_l$, then the \textit{general edge-degree} of $e_i$ is defined to be $d_G(e_i) = d_{G(v_k)}(e_i) + d_{G(v_l)}(e_i)=d_G(v_k)+d_G(v_l)-2$. (Also see \cite{AV1}).}
\end{defn}

\ni Similarly, the edge-degree of a vertex of a given graph $G$ is defined as follows.

\begin{defn}\label{Def-2.3}{\rm
For a graph $G$ with edges $E(G) = \{e_i: 1 \le i \le \epsilon\}$ the \textit{edge-degree of a vertex} $v_i$, denoted by $d_{G(e)}(v_i)$, is defined as $d_{G(e)}(v_i) = \sum\limits_{(e_l \to v_i)}d_{G(v_i)}(e_l)$.}
\end{defn}

Using the above definitions, we introduce the notion of the total edge-degree of a graph as follows.

\begin{defn}\label{Def-2.4}{\rm 
The \textit{total edge-degree} of a graph G, denoted by $d^t_{G(e)}(G)$, is defined by $d^t_{G(e)}(G) = \sum\limits_{v_i \in V(G)}d_{G(e)}(v_i)$.}
\end{defn}

Invoking these definitions given above, it can be noticed that both $d^t_{G(e)}(G)$ and $d_{G(e)}(v_i)$ are single values, while $d_{G(v_k)}(e_j) \le d_{G(v_m)}(e_j), (e_j \to v_k), (e_j \to v_m)$.

\begin{ill}{\rm 
The graphs having three edges $e_1, e_2, e_3$ are graphs $P_4, C_3$, and $K_{1,3}$. The corresponding edge-set graphs on the vertices $v_{1,1} =\{e_1\}, v_{1,2} = \{e_2\}, v_{1,3} = \{e_3\}, v_{2,1}=\{e_1,e_2\}, v_{2,2}=\{e_1, e_3\}, v_{2,3}=\{ e_2,e_3\}, v_{3,1}=\{e_1,e_2,e_3\}$ are depicted below.}
\end{ill} 

\begin{figure*}[h!]
    \centering
    \begin{subfigure}[b]{0.5\textwidth}
        \centering
        \includegraphics[height=1.75in]{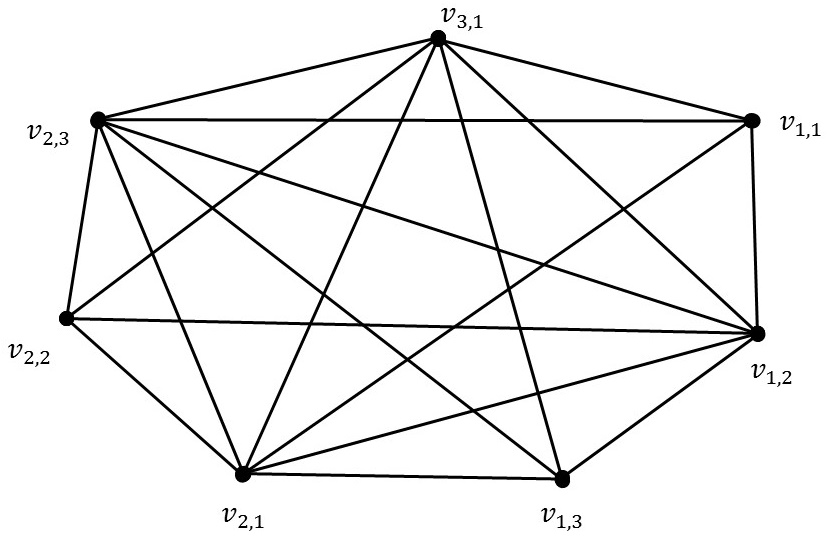}
        \caption{$\G_{P_4}$}
    \end{subfigure}%
    ~ 
    \begin{subfigure}[b]{0.5\textwidth}
        \centering
        \includegraphics[height=1.75in]{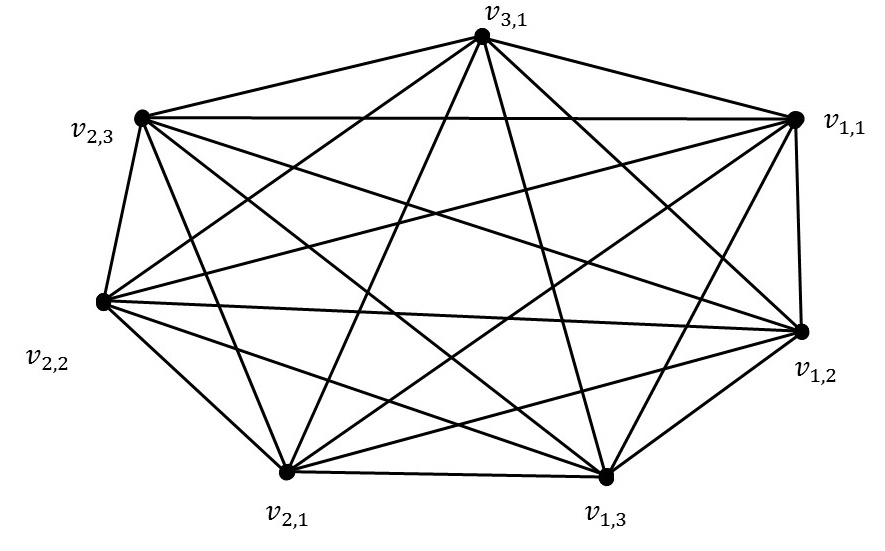}
        \caption{$\G_{C_4}= \G_{K_{1,3}}=K_7$}
    \end{subfigure}
    \caption{}
\end{figure*}

We notice that both $\G_{C_3}$ and $\G_{K_{1,3}}$ are complete graphs. This observation follows from results to follow. The following is a straight forward result and an important property of edge-set graphs as well.

\begin{prop}
The induced graph $\langle \{v_{1,i}: 1\le i \le \epsilon\}\rangle$ of the edge-set graph $\G_G$ is isomorphic to the line graph of $G$.
\end{prop}
\begin{proof}
Let $\G_G$ the edge-set graph of a given graph $G$. Then, every vertex $v_{1,i}$ of $\G_G$ corresponds to the singleton subset $\{e_i\}$ of $E(G)$, where $1\le i\le \epsilon$ and two vertices $v_{1,i}$ and $v_{1,j}$ of $\G_G$ are adjacent in $\G_G$ if and only if the corresponding edges $e_i$ and $e_j$ are adjacent in $G$. Hence, the induced graph $\langle v_{1,i}; 1\le i \le \epsilon\rangle$ of the edge-set graph $\G_G$ is isomorphic to $L(G)$.
\end{proof}

Analogous to the corresponding theorem on the relation between the minimum and maximum degree of vertices in a set-graph, we propose the following theorem.

\begin{thm}\label{Thm-2.6}
For the edge-set graph $\G_G$, $\epsilon \ge 1$ corresponding to a given graph $G$, we have $\Delta(\G_G) = 2(2^{\epsilon-1}-1)$ and $\delta(\G_G) \ge 2(\epsilon-1)$.
\end{thm}
\begin{proof} First part of the Theorem follows immediately by Theorem \ref{T-DSG2}, we have $\Delta(\G_G) = 2(2^{\epsilon-1}- 1)$.

Next, note that the graph $G$ with $\epsilon$ edges, with $\min\{d^t_{G(e)}(G)\} = 2(\epsilon-1)$ is the path $P_{\epsilon +1}$. It can be verified that for $P_n, 2\le n\le 4$, $\delta(\G_{P_{\epsilon + 1}})= 2(\epsilon-1)$, where $\epsilon=1,2,3$. For $P_5$, we have $d^t_{P_5(e)}(P_5)=2(4 -1)=6<8=\delta(P_{5, E^{(4)}})$. In a similar way, by mathematical induction, the result $2(\epsilon-1) \le \delta(P_{(\epsilon + 1),E})\le \delta(\G_G)$, $\epsilon \ge 1$ follows. 
\end{proof}

The following theorem describes the condition required for an edge-set graph to be a complete graph.

\begin{thm}\label{Thm-2.7}
For $\epsilon \ge 4$, the edge-set graph of a given graph $G$ is complete if and only if $G \cong K_{1,\epsilon}$. In other words, the edge-set graph of a graph $G$, with more than three edges, is a complete graph if and only if $G$ is a star graph.
\end{thm}
\begin{proof}
First consider the graph $G=K_{1,4}$, whose central vertex is denoted by $u$ and the other vertices are denoted by $v_1,v_2,v_3,v_4$ respectively. Let the edges of $G$ be given by $e_i=uv_i;~1\le i\le 4$. Clearly, the number of non-empty subsets of $E(K_{1,4})$ is given by $2^4 -1$ and each subset has at least one edge $e_i \in E(K_{1,4})$ adjacent to at least one edge $e_j \in E(K_{1,4})$ in each of the other subsets. Hence, pairwise there exists an edge $v_{s,i}v_{t,j}$ for all $s$-element and $t$-element subsets of $E(K_{1,4})$. Hence, $K_{(1,4),E^{(4)}}$ is complete. Now, assume that the results holds for $G = K_{1,k}$, $k$ being a positive integer greater than $4$. Now, consider the graph $G=K_{1,(k+1)}$ with $u$ as the central vertex and $v_{k+1}$ the new vertex in $K_{1,(k+1)}$ that is not in $K_{1,k}$. Then, the new edge $e_{k+1}=uv_{k+1}$ is adjacent to all other edges in $G$ and hence the vertices in the edge-set graph with the set-label containing the element $e_{k+1}$ will also be adjacent to all other vertices of the edge-set graph $\G_{K_{1,k+1}}$. Therefore, $\G_{K_{1,k+1}}$ is also a complete graph. Hence, by mathematical induction, the result holds for all $K_{1,\epsilon}, \epsilon \ge 4$.

Conversely, assume the edge-set graph $\G_G, ~|V(\G_G)| = 2^\epsilon -1 \ge 15$ is complete. Therefore, $\G_G - v_{\epsilon,1}$ on $2^\epsilon-2=\frac{1}{2}(2^{\epsilon +1}-1)$ vertices is complete as well. Applying the inverse of Definition \ref{D-ESG}, to $\G_G-v_{\epsilon,1}$ results in a smallest graph $K_{1, \epsilon}$. Assume the edge-set graph $\G_G$ also corresponds to $K_{1,\epsilon} + v_iv_j$ with $v_iv_j$ an arbitrary additional edge. Now, we must consider an edge-set graph on $2^{\epsilon+1} -1 > \frac{1}{2}(2^{\epsilon +1}-1) +1$ edge-subsets. The aforesaid edge-set graph is not isomorphic to the initial graph $\G_G, |V(\G_G)| = 2^\epsilon-1 \ge 15$. Hence, $K_{1,\epsilon}$ is the unique graph corresponding to the complete edge-set graph, $\G_G, |V(\G_G)| = 2^\epsilon -1 \ge 15 \Rightarrow \epsilon \ge 4$. This completes the proof. 
\end{proof}

It can be followed that the edge-set graphs corresponding to  $K_{1,1}, K_{1,2}$ and $K_{1,3}$ are complete. Furthermore, for any edge $e_i$ in the cycle $C_n, n \ge 4$ there is exactly $n-3$ edges which are not adjacent to $e_i$. By Theorem 2.3, $\G_{C_n}$ is not complete. Hence, $C_3$ is the only cycle for which the corresponding edge-set graph is complete.

\begin{thm}\label{Thm-2.8}
For $\epsilon = n \ge 2$, we have $\sum\limits_{v_{i,j} \in V(\G_G)}d_{\G_G}(v_{i,j}) > \sum\limits_{v_{i,j} \in V(G_{A^{(n)}})}d_{G_{A^{(n)}}}(v_{i,j})$.
\end{thm}
\begin{proof}
Since paths have the minimally connectivity, we first prove the result for paths. It is to be noted that For $\epsilon=n=2$, we have $\G_{P_3} \cong K_3$ and $G_{A^{(2)}} \cong P_3$. Hence,$\sum\limits_{v_{i,j} \in V(\G_{P_3})}d_{\G_{P_3}}(v_{i,j}) \ge \sum\limits_{v_{i,j} \in V(\G_{P_3})}d_{\G_{P_3}}(v_{i,j})  = 6 > 4 = \sum\limits_{v_{i,j} \in V(G_{A^{(2)}})}d_{G_{A^{(2)}}}(v_{i,j})$. 

Assume the result holds for all paths $P_{k+1}$ hence, by implication it holds for all graphs $G$ on $\epsilon = n = k$ edges.

Now consider $P_{k+2} = P_{k+1} + e_{k+1}$. Note that that in the set-graph $G_{A^{(k+1)}}$, singleton subsets are pairwise non-adjacent (see \cite{KC2S}) while in an edge-set graph some pairs of singleton edge-subsets are adjacent if and only if the corresponding edges are adjacent in $G$. Hence, it can be noted that $d_{\G_{P_{k+2}}}(v_{1,k+1})=d_{G_{A^{(k+1)}}}(v_{1,k+1}) +1 > d_{G_{A^{(k+1)}}}(v_{1,k+1})$. Hence, $\delta(\G_{P_{k+2}}) > \delta(G_{A^{(k+1)}})$. Similarly, we have $d_{\G_{P_{k+2}}}(v_{s,i}) \ge d_{G_{A^{(k+1)}}}(v_{s,i})$, $\forall s,i$. We also note that $G_{A^{(k+1)}}$ has a unique vertex having degree equal to $\Delta(G_{A^{(k+1)}})$ (see \cite{KC2S}) when the edge-set graph $\G_{P_{k+2}}$ has four such vertices. Hence,  we have $\sum\limits_{v_{i,j} \in V(\G_{P_{k+2}})}d_{\G_{P_{k+2}}}(v_{i,j})> \sum\limits_{v_{i,j} \in V(G_{A^{(k+1)}})}d_{G_{A^{(k+1)}}}(v_{i,j})$. Therefore, the result follows for all paths by mathematical induction.

Since all other connected graphs have higher connectivity than paths, the result, $\sum\limits_{v_{i,j} \in V(\G_G)}d_{\G_G}(v_{i,j}) > \sum\limits_{v_{i,j} \in V(G_{A^{(n)}})}d_{G_{A^{(n)}}}(v_{i,j})$, holds for any connected graphs.
\end{proof}

In \cite{KC2S}, it is proved that there exists a unique vertex $v_{n,1}$ in a set-graph $G_{A^{(n)}}$ having degree $\Delta(G_{A^{(n)}})$. In edge-set graphs, the analogous result does not hold. In order to discuss the existence of multiple vertices with the maximum degree in an edge-set graph,we need the following definition.

\begin{defn}{\rm 
A \textit{connected edge dominating set} of a graph $G$ is an edge dominating set $X \subseteq E(G)$ where the induced graph $\langle X \rangle$ is a connected subgraph of $G$. The minimum number of elements in a connected edge dominating set of a graph $G$ is called the \textit{connected edge domination number} (CED-number) of $G$.}
\end{defn}

\ni In view of this definition, we have the following theorem.

\begin{thm}\label{Thm-2.9}
Let $X$ be a connected edge dominating set of a graph $G$ with $\epsilon \ge 2$ and let $V_X(\G_G)$ be the the subset of vertex set of the edge-set graph $\G_G$ defined by $V_X(\G_G) = \{v_{s,i}: E_{s,i} \in \cS, X\subseteq E_{s,i}\}$. Then, we have $d_{\G_G}(v_{s,i}) = \Delta(\G_G), \forall v_{s,i} \in V_X(\G_G)$. 
\end{thm}
\begin{proof}
By the choice of $X$, it is the set-labels of some element in $V_X(\G_G)$. Consider any edge $e_j \in E(G)$. Then, we consider the following cases.

\ni \textit{Case-1:} If $e_j \in X$, then there exists a vertex, say $v_{l,k}$, in $\G_G$ corresponding to $X$, which is adjacent to all other vertices corresponding to the edge-subsets  $E_{s,i}$ containing the edge $e_j$. Since $X$ is an edge dominating set onto itself, this argument holds for all edges in $X$.

\ni \textit{Case-2:} Since $X$ is an edge dominating set, all edges in $E(G)-X$ are adjacent to some edges in $X$. If $e_j \in E(G)-X$, then a similar argument holds.

\ni From the above two cases, we can see that any arbitrary vertex $v_{l, k}$ of $V_X(\G_G)$ is adjacent to all vertices in $\G_G$. Hence, we have $d_{\G_G}(v_{s,i})=\Delta(\G_G)$.
\end{proof}

\begin{rem}{\rm 
The edge-set $E(G)$ is the set-label of some vertex in $V(\G_G)$ and for any edge $e_i\in E(G),i\le \epsilon$, the edge-subsets $E(G)-e_i$ are connected edge dominating sets in $G$. Hence, we have $d_{\G_G}(v_{\epsilon,1})=\Delta(\G_G)$ and $d_{\G_G}(v_{\epsilon-1,j})= \Delta (\G_G)$.}
\end{rem}



By Theorem \ref{Thm-2.7}, we note that the edge-set graphs corresponding to paths, cycles and complete graphs are not complete graphs. Hence, determining the maximum and minimum degrees of the corresponding edge-set graphs arouses much interest in this context. 

Let us now introduce the notion of connected edge domination index of a given graph as follows.
 
\begin{defn}{\rm 
The \textit{connected edge domination index} (CED-index) of a graph $G$ is defined as the number of smallest connected edge dominating sets in $G$ and is denoted by $\fE(G)$.}
\end{defn}

The CED-indices of paths, cycles and complete graphs are determined in the following results.

\begin{prop}\label{Ppn-3.1(a)}
The  CED-index of a path $P_n$ is $1$.
\end{prop}
\begin{proof}
The path $P_n$, $n \ge 3$ has exactly three connected edge dominating sets, namely $\{e_1, e_2, e_3, \ldots, e_{n-2}\}$, $\{e_2, e_3, e_4, \ldots, e_{n-1}\}$ and $\{e_2, e_3, e_4, \ldots, e_{n-2}\}$. Clearly, the  third set is the smallest connected edge dominating set. That is, number of minimal connected edge dominating set $\fE(P_n) = 1$.
\end{proof}

\begin{prop}\label{Ppn-3.1(b)}
The  CED-index of a cycle $C_n$ is $n$.
\end{prop}
\begin{proof}
Consider the set of edge pairs $\fS =\{\{e_i,e_{i+1}\}: 1 \le i \le n\}$ of the cycle $C_n$, in the sense that $e_{n+1}=e_1$.  Clearly, $E(C_n) - \{e_i,e_j\}$, $\forall \{e_i, e_j\} \in \fS$ are all smallest connected edge dominating sets of $C_n$. Hence, we have $\fE(C_n) = n$.
\end{proof}

\begin{prop}\label{Ppn-3.1(c)}
The  CED-index of a complete graph is $\frac{n}{2}(n-1)!$.
\end{prop}
\begin{proof}
It is known that $K_n, n \ge 3$ has $\frac{1}{2}(n-1)!$ distinct Hamiltonian cycles (see \cite{KZ}). Also, a minimal connected edge dominating set of a Hamiltonian cycle of $K_n$ is also a smallest connected edge dominating set of $K_n$. By Proposition \ref{Ppn-3.1(b)}, we have $\fE(K_n) = \frac{n}{2}(n-1)!$.
\end{proof}

In the proof of Proposition \ref{Ppn-3.1(a)} it was noted that $P_n$, $n \ge 3$ has exactly three connected edge dominating sets hence, $\G_{P_n}$ has exactly four vertices with maximum degree. Let $\fM(\G_G)$ denote the number of vertices of $\G_G$ having the maximum degree $\Delta(\G_G)$. 

The number of maximum degree vertices in the edge-set graphs corresponding to the graphs $C_n$ and $K_n$ are determined in the following results.

\begin{prop}\label{Ppn-3.2(a)}
The number of maximum degree vertices in the edge-set graph of $C_n$ is $3n+1$. 
\end{prop}
\begin{proof}
Without loss of generality, any smallest connected edge dominating set of $C_n$ say, $\{e_2,e_3.e_4, \ldots, e_{n-1}\}$ can be extended to a larger connected edge dominating set, say $\{e_2,e_3,e_4, \ldots, e_{n-1}, e_n\}$ or $\{e_1,e_2,e_3, \ldots, e_{n-1}\}$. Hence, an additional $2n$ such extended edge sets exists. $E(C_n)$ is also a connected edge dominating set as well so invoking Proposition \ref{Ppn-3.1(b)}, implies $\fM(C_n) = 3n + 1$.
\end{proof}

\begin{prop}
The number of maximum degree vertices in the edge-set graph of $K_n$ is $\frac{1}{2}(3n+1)(n-1)!$. 
\end{prop}
\begin{proof}
Similar to the reasoning in the above proposition, each of the $\frac{1}{2} (n-1)!$ Hamiltonian cycles in $K_n$ is a connected edge dominating set of $K_n$. Also, each of the smallest connected edge dominating sets can be extended in the manner shown in Proposition \ref{Ppn-3.2(a)}. Hence, $\fM(C_n) = 3 \frac{n}{2} (n-1)! + \frac{1}{2} (n-1)! = (\frac{3n+1}{2}) (n-1)!$. 
\end{proof}

The next result is another obvious but important property of edge-set graphs, which leads to an interesting result in respect of regular graphs.

\begin{thm}\label{Thm-3.3}
Let $e$ be an edge of a graph $G$ having the minimum edge degree in $G$. Then, the vertex $v$ in $\G_G$ corresponding to the singleton edge-subset $\{e\}$ has minimum degree in $\G_G$.
\end{thm}
\begin{proof}
Assume that the edge $e$ has minimum edge-degree $d_G(e) = l$. Label the edges adjacent to $e$ as $f_1,f_2,f_3, \ldots, f_l$. From the definition edge-set graphs, it follows that a vertex $v$ is adjacent to $v_{s,i}$ if and only if some of these edges $f_i$ belongs to $E_{s,i}$. This number of edges that are adjacent to $e$ in $G$ is equal to the degree of $v$ in $\G_G$. 

Now, assume there exists a vertex $v_{t,j}$ with $d_{\G_G}(v_{t,j}) = \delta(\G_G) < d_{\G_G}(v)$. It implies that by Definition \ref{D-ESG}, the vertex $v_{t,j}$ is adjacent to $v_{k,m}$ if and only if some or all edges $g_1, g_2,g_3, \ldots,g_s$, which are in $E_{k,m}$ and an edge $e^{\ast} \in E_{t,j}$ is adjacent to an edge $g_i \in E_{k,m}$, with $s < l$ and $d_G(g_i) < s$. This contradicts the assumption that $d_G(e)$ has minimum degree in $G$. Hence, $d_{\G_G}(v) = \delta(\G_G)$.
\end{proof}

As an immediate consequence of Theorem \ref{Thm-3.3}, we observe that the edge-set graph $\G_G$ corresponding to an $r$-regular graph $G$ has $\epsilon$ vertices, $v_{1,1}, v_{1,2}, v_{1,3}, \ldots, v_{1,\epsilon}$, with degree equal to $\delta(\G_G)$. Also, we note that an edge $e$ is adjacent to exactly $d_G(e)$ edges in $G$. Therefore, there are exactly $\epsilon-d_G(e)$ edges that are not adjacent to the edge $e$ in $G$. This leads us to the following result.

\begin{prop}\label{Ppn-3.4}
The degree of the vertex $v_{1,i}$ corresponding to $\{e_i\}$ in $\G_G$ is given by $d_{\G_G}(v_{1,i})=2^\epsilon-2^{\epsilon-d_G(e_i)}$.
\end{prop}
\begin{proof}
The edge-set graph $\G_G$ has $2^\epsilon - 1$ vertices. Exactly $2^{\epsilon-d_G(e_i)} -1$ edge-subsets do not contain an edge adjacent to $e_i$ as an element. Hence, exactly $2^\epsilon - 2^{\epsilon-d_G(e_i)}$ edge-subsets do contain an edge adjacent to $e_i$ as an element. Therefore, the degree of the vertex $v_{1,i}$ in $\G_G$ is $2^\epsilon-2^{\epsilon-d_G(e_i)}$.
\end{proof}

As a consequence of Proposition \ref{Ppn-3.4}, an $r$-regular graph $G$ on $\nu$ vertices has $\delta(\G_G) =2^{\frac{1}{2}r\nu}(1-\frac{1}{2^{\frac{1}{2}(r-1)}})$.







\ni Let us now recall the well-known theorem on Eulerian graphs.

\begin{thm}\label{T-EUG}
{\cite{FH}} A connected graph $G$ is Eulerian if and only if $G$ has no odd degree vertices.
\end{thm}

In view of Theorem \ref{T-EUG}, we establish Eulerian nature of the edge-set graphs in the following theorem.

\begin{thm}
The edge-set graph of any connected graph $G$ is an Eulerian graph.
\end{thm}
\begin{proof}
Let $\G_G$ is the edge-set graph of a given graph $G$ with $\epsilon$ edges. Note that the number of subsets of $E(G)$ containing the same $r$ edges is $2^{\epsilon-r}$, where $r\le \epsilon$, which is always an even integer. 

Let $e_i$ be an arbitrary edge of $G$. Since $G$ is a connected graph, the edge $e_i$ is adjacent to at least one edge in $G$. Without loss of generality, let the edge $e_i$ is adjacent to $r$ edges in $G$, where $r\ge 1$. Let $e_{j_1}, e_{j_2}, e_{j_3}, \ldots, e_{j_r}$ be the $r$ edges in $E(G)$ that are adjacent to $e_i$ in $G$ Then, every vertex $v$ in $\G_G$, whose set-label contain the element $e_i$ must be adjacent to the vertices of $\G_G$, whose set-labels contain at lease one element $e_{j_l}$, where $1\le l\le r$. Let $E_l$ denotes the subset of $E(G)$ containing the element $e_{j_l}$. Then, by the principle of inclusion and exclusion, we have the number of vertices in $G_G$ that are adjacent to the vertex $v$ is $|\bigcup\limits_{l=1}^{r}E_{j_l}|=\sum\limits_L(-1)^{|L|-1}|\bigcap\limits_{l\in L}E_{j_l}|$, where $\emptyset\ne L\subseteq \{1,2,3,\ldots, r\}$. Note that every term in the above series is an even integer and as a result $d_{\G_G}(v)$ is always an even integer. Since $e_i$ is an arbitrary edge of $G$, $v$ is also an arbitrary vertex (corresponding to the choice of $e_i$) of $G_G$ and hence every vertex of $G_G$ is an even degree vertex. Then, by Theorem \ref{T-EUG}, $G_G$ is an Eulerian graph. 
\end{proof}

\section{Conclusion and Scope for Further Studies}

So far we have discussed the properties and characteristics of a new types of graphs called edge-set graphs derived from the edge sets of given graphs. This study is only a preliminary study on edge-set graphs and several related areas still remain to be explored. In this paper, we have established certain results for paths, cycles and complete graphs only. Extending these results to more standard graphs are much interesting and challenging problems. More fruitful results are expected on certain edge-set graphs on the effective utilisation of Definitions \ref{Def-2.2}, Definitions \ref{Def-2.3} and Definitions \ref{Def-2.4} to more graph classes.


It seems to be worthy and promising for future studies on the relationship between $\G_G$, $\G_{G+e}$ and, where $G+e$ is the graph obtained by joining two non-adjacent vertices in $G$ by the edge $e$. The results related to the edge-set graphs of the graphs obtained by different graph operations are also open.  Another important area that offer much for further investigations is the study on the edges of the edge-set graphs and set-graphs of given graphs. More studies are possible on the comparison of the set-graphs and edge-set graphs derived from the edge sets of different graphs. All these facts highlight the wide scope for further research in this area.


\begin{thebibliography}{25}

\bibitem{AV1} S. Arumugam and S. Velammal, {\em Edge Domination in Graphs}, Taiwanese Journal of Mathematics, \textbf{2}(2)(1998), 173-179.

\bibitem{BM1} J. A. Bondy and U. S. R. Murty, {\bf Graph Theory with Applications}, Macmillan Press, London, 1976.
	
\bibitem{CL1} G. Chartrand and L. Lesniak, {\bf Graphs and Digraphs}, CRC Press, 2000.
	
\bibitem{GY1} J. T. Gross and J. Yellen, {\bf Graph Theory and its Applications}, CRC Press, 2006. 
	
\bibitem{FH}  F. Harary, {\bf Graph Theory}, Addison-Wesley, 1994.
	
\bibitem{KC2S} J. Kok, K. P. Chithra, N. K. Sudev and C. Susanth, {\em A Study on Set-Graphs}, International Journal of Computer Applications, {\bf 118}(7)(2015), 1-5., DOI: 10.5120/20754-3173.
	

\bibitem{KZ} J Kratochvil, D. Zeps, {\em On the Number of Hamilton cycles in Triangulations}, Journal of Graph Theory, {\bf 12}(2)(1988), 191-194.  

\bibitem{MM} T. A. McKee and F. R. McMorris, {\bf Topics in Intersection Graph Theory}, SIAM, Philadelphia, 1999.
	
\bibitem{KHR} K. H. Rosen, {\bf Handbook of Discrete and Combinatorial Mathematics}, CRC Press, 2000. 

\bibitem{VP} S. K. Vaidya and R. M. Pandit, {\em Edge Domination in Some Path and Cycle Related Graphs}, ISRN Discrete Mathematics, 2014 (2014), Article ID: 975812, 1-5., DOI: 10.1155/2014/975812.

	
\bibitem{DBW} D. B. West, {\bf Introduction to Graph Theory}, Pearson Education Inc., 2001.
	
\end{thebibliography}
\end{document}